\documentclass{amsart}

\usepackage{amssymb,amsmath,bbm}
\usepackage[margin=1in]{geometry}
\usepackage[backref=page]{hyperref}
\usepackage[width=12cm,format=plain,labelfont=sf,bf,small,textfont=sf,up,small,justification=justified,labelsep=period]{caption}
\usepackage{graphicx}
\usepackage{caption}
\usepackage{subcaption}
\usepackage{framed}
\usepackage{xcolor}
\usepackage{tabularx}

\hypersetup{
    colorlinks=true,       
    linkcolor=blue,          
    citecolor=blue,        
    filecolor=blue,      
    urlcolor=blue           
}

\newcommand{\RR}{\mathbb{R}}

\newcommand{\rank}{\textup{rank}\,}

\newcommand{\rankpsd}{\textup{rank}_{\textup{psd}}\,}

\renewcommand{\S}{\mathcal S} 

\newtheorem{theorem}{Theorem}

\newtheorem{lemma}[theorem]{Lemma}

\newtheorem{proposition}[theorem]{Proposition}
\theoremstyle{definition}
\newtheorem{definition}[theorem]{Definition}

\theoremstyle{remark}
\newtheorem{remark}[theorem]{Remark}

\title{Rational and real positive semidefinite rank can be different}

\author[J. Gouveia]{Jo{\~a}o Gouveia}
\address{CMUC, Department of Mathematics,
  University of Coimbra, 3001-454 Coimbra, Portugal}
\email{jgouveia@mat.uc.pt}

\author[H. Fawzi]{Hamza Fawzi} \address{Laboratory for Information and
  Decision Systems (LIDS), Massachusetts Institute of Technology,
  Cambridge, MA 02139, USA} \email{hfawzi@mit.edu}

\author[R.Z. Robinson]{Richard Z. Robinson}
\address{Department of Mathematics, University of Washington, Box
  354350, Seattle, WA 98195, USA} \email{rzr@uw.edu}

\begin{document}

\begin{abstract}
Given a nonnegative matrix $M$ with rational entries, we consider two quantities: the usual positive semidefinite (psd) rank, where the matrix is factored through the cone of real symmetric psd matrices, and the rational-restricted psd rank, where the matrix factors are required to be rational symmetric psd matrices.  It is clear that the rational-restricted psd rank is always an upper bound to the usual psd rank.  We show that this inequality may be strict by exhibiting a matrix with psd rank four whose rational-restricted psd rank is strictly greater than four.
\end{abstract}

\maketitle

The \emph{positive semidefinite (psd) rank} of a matrix was introduced in \cite{gouveia2011lifts} (see also \cite{fiorini2012lowerbound}).  In this note, we answer a basic structural question about the psd rank: if a nonnegative matrix $M$ has only rational entries, can the psd rank of $M$ always be achieved by a factorization using only rational matrices?  We answer this question negatively by providing an example of a rational matrix with psd rank four such that every psd factorization of size four uses irrational numbers.  
Note that the analogous question for the \emph{nonnegative rank} of a matrix was posed by Cohen and Rothblum in \cite{cohen1993nonnegative}.  It was shown in \cite{cohen1993nonnegative} that all rational matrices with nonnegative rank two admit a rational nonnegative factorization, but the question for general nonnegative matrices remains open.

Let $\S^k_+$ denote the cone of real symmetric $k \times k$ psd matrices.  The psd rank of a nonnegative matrix is defined as follows:
\begin{definition}
\label{def:psdfactorization}
Given a nonnegative matrix $M \in \RR^{p\times q}_+$, a \emph{psd factorization} of $M$ of size $k$ is a collection of psd matrices $A_1,\dots,A_p \in \S^k_+$ and $B_1,\dots,B_q \in \S^k_+$ such that $M_{ij} = \langle A_i, B_j \rangle$ for all $i=1,\dots,p$ and $j=1,\dots,q$, where the inner product is the standard trace inner product on symmetric matrices.\\
The \emph{psd rank} of $M$, denoted $\rankpsd M$, is the smallest integer $k$ for which $M$ admits a psd factorization of size $k$.
\end{definition}

The proof of our example will require a lemma about rational psd matrices of rank one.  Any rank one psd matrix has the form $\mathbf{v} \mathbf{v}^T$ for some vector $\mathbf{v}$.  Let $\phi$ denote the map taking the vector $\mathbf{v}$ to the psd matrix $\mathbf{v} \mathbf{v}^T$.  Then we have the following.

\begin{lemma}
\label{lem:rational-rankone}
If the matrix $\phi (\mathbf{v})$ is composed of only rational entries, then $\mathbf{v}$ has the form $\alpha \mathbf{q}$ where $\alpha$ is a real scalar and $\mathbf{q}$ is a rational vector.
\end{lemma}

\begin{proof}
Suppose that $\mathbf{v}$ is a nonzero vector (else the conclusion is immediate).  Then, without loss of generality, we may assume that the first coordinate $v_1$ is nonzero.  Since $v_1^2$ is an entry in the matrix $\phi (\mathbf{v})$, it must be rational. Hence, the matrix $\frac{1}{v_1^2} \phi (\mathbf{v})$ is also rational.  By looking at the first row of this matrix, we see that the vector $\left( 1, \frac{v_2}{v_1}, \frac{v_3}{v_1}, \ldots,\frac{v_r}{v_1} \right)$ is rational.  Now we just scale this rational vector by $v_1$ to finish the proof.
\end{proof}

Our candidate matrix $M$ is the 
$8 \times 6$ matrix shown in Figure~\ref{fig:cuboid_slack}.
Readers who are familiar with slack matrices may be interested to know that $M$ arises as a slack matrix of the polytope with vertices $(0,0,0)$, $(1,0,0)$, $(0,1,0)$, $(1,2,0)$, $(0,0,1)$, $(1,0,1)$, $(0,1,1)$, and $(1,2,1)$.  Readers who are not familiar with slack matrices need not worry, as we will refrain from using any results about slack matrices in the proofs.

\begin{figure}
  \qquad \qquad
	\begin{subfigure}[h]{0.4\textwidth}
		\[ M = \left( \begin{array}{cccccc}
		0 & 0 & 2 & 1 & 0 & 1 \\
		1 & 0 & 0 & 2 & 0 & 1 \\
		0 & 1 & 2 & 0 & 0 & 1 \\
		1 & 2 & 0 & 0 & 0 & 1 \\
		0 & 0 & 2 & 1 & 1 & 0 \\
		1 & 0 & 0 & 2 & 1 & 0 \\
		0 & 1 & 2 & 0 & 1 & 0 \\
		1 & 2 & 0 & 0 & 1 & 0 \end{array} \right)\]
	\end{subfigure}
	\begin{subfigure}[h]{0.5\textwidth}
		\includegraphics[width=0.5\textwidth]{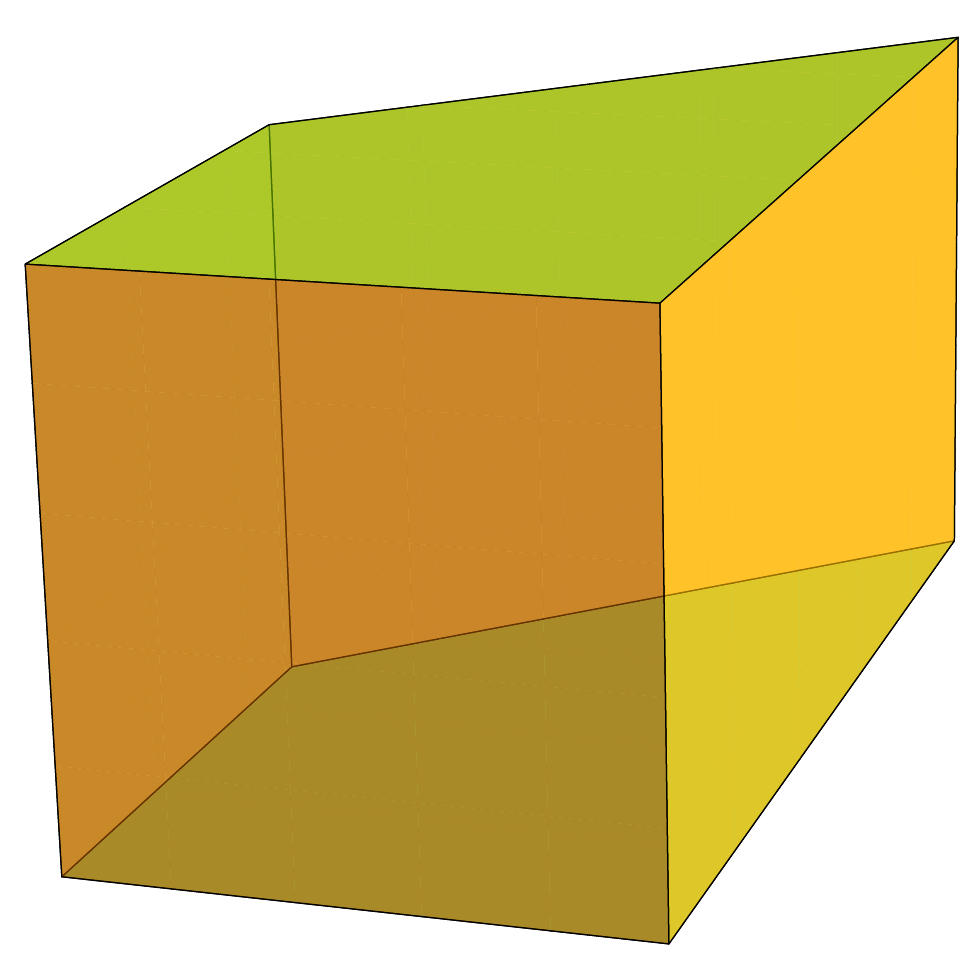}
	\end{subfigure}
\caption{Our example matrix is a slack matrix of a three dimensional polytope.}
\label{fig:cuboid_slack}
\end{figure}

During our analysis of this example, we will require a few results about psd rank found in the literature.  We summarize these results in the following proposition.

\begin{proposition}\label{prop:quotes}
\quad 
\begin{enumerate}
\item \cite[Prop. 5]{gouveia2011lifts} If $\sqrt{A}$ is an entry-wise square root of $A$, then $\rankpsd A \leq \rank \sqrt{A}$.
\item \cite[Cor. 4.8]{leetheis2012support}, \cite[Prop. 2.6]{gouveia2013polytopes} If $A$ contains a $k \times k$ triangular submatrix $T$, then $\rankpsd A \geq k$.  Furthermore, in a psd factorization of $A$ of size $k$, the factor corresponding to the row (or column) of $T$ with $k-1$ zeros must have rank one.
\end{enumerate}
\end{proposition}

Now we begin our analysis of the matrix $M$.

\begin{lemma}\label{lem:Mprops}
We have that $\rankpsd M = 4$ and any psd factorization of $M$ of size four uses only rank one factors.
\end{lemma}
\begin{proof}
One can verify that the all-nonnegative entry-wise square root of $M$ has usual rank four. Thus Proposition~\ref{prop:quotes} says that $\rankpsd M \leq 4$. Consider the submatrix of $M$ indexed by rows 1, 5, 7, and 8 and columns 1, 2, 5, and 6.  This submatrix is triangular so Proposition~\ref{prop:quotes} tells us two things:  First, we have that $\rankpsd M \geq 4$ and, hence, $\rankpsd M = 4$.  Second, the factors corresponding to the first row and first column in a psd factorization of $M$ of size four must always be rank one.  It is easy to verify by inspection that for every row and column of $M$ we can find a $4 \times 4$ triangular submatrix such that the row or column in question has three zeros in that submatrix.  Thus, repeatedly applying the proposition completes the proof.
\end{proof}
\begin{remark}
Note that Lemma \ref{lem:Mprops} is actually a consequence of \cite[Prop. 3.2]{gouveia2013polytopes} since our polytope has minimal psd rank (equal to the ambient dimension plus one) and thus any psd factorization must consist entirely of rank-one factors.
\end{remark}

The next proposition shows that no rational psd factorization of $M$ can have size four.

\begin{proposition}
We have that $\rankpsd M = 4$, but there does not exist a psd factorization of size four using only rational matrices. 
\end{proposition}

\begin{proof}
Suppose, by way of contradiction, that $\left(A_1,\ldots,A_8,B_1,\ldots,B_6 \right)$ is a psd factorization of $M$ of size four that uses only rational matrices.  By Lemma~\ref{lem:Mprops}, each matrix must be rank one.  Thus, there exist vectors $\mathbf{a_1},\ldots,\mathbf{a_8}$ and $\mathbf{b_1},\ldots,\mathbf{b_6}$ such that $A_i = \phi(\mathbf{a_i})$ and $B_j = \phi (\mathbf{b_j})$.  
Furthermore, by the properties of the trace, we must have that $M_{ij} = \langle A_i, B_j \rangle = \langle \mathbf{a_i}, \mathbf{b_j} \rangle^2$.
Thus, the matrix whose $(i,j)$th entry is given by $\langle \mathbf{a_i}, \mathbf{b_j} \rangle$ is an entry-wise square root of $M$, which we denote by $S$.  By looking at the submatrix generated by the first two rows and the fourth and sixth columns, we see that $S$ contains a submatrix $\widetilde{S}$ of the form
$$ \left( \begin{array}{cc}
\pm 1 & \pm 1 \\ \pm \sqrt{2} & \pm 1 \end{array} \right)$$
where there is ambiguity on the sign of each entry.

Now by Lemma \ref{lem:rational-rankone}, each $\mathbf{a_i}$ and $\mathbf{b_j}$ must be a rational vector scaled by a nonzero real number.  Hence, there must exist nonzero real numbers $\alpha_1, \alpha_2, \beta_1, \beta_2$ such that the matrix resulting from the product
$$ \left( \begin{array}{cc}
\alpha_1 & 0 \\ 0 & \alpha_2 \end{array} \right)
\cdot
\widetilde{S}
\cdot
\left( \begin{array}{cc}
\beta_1 & 0 \\ 0 & \beta_2 \end{array} \right) = 
\left( \begin{array}{cc}
\pm \alpha_1 \beta_1 & \pm \alpha_1 \beta_2 \\ \pm \sqrt{2} \alpha_2 \beta_1 & \pm \alpha_2 \beta_2 \end{array} \right) $$
is rational.  It is easy to see that if $\alpha_1 \beta_1$, $\alpha_1 \beta_2$, and $\alpha_2 \beta_2$ are rational, then $\alpha_2 \beta_1$ must also be rational, which results in a contradiction.
\end{proof}

\bibliography{psdrank}

\newcommand{\etalchar}[1]{$^{#1}$}
\begin{thebibliography}{FMP{\etalchar{+}}12}

\bibitem[CR93]{cohen1993nonnegative}
J.E. Cohen and U.G. Rothblum.
\newblock Nonnegative ranks, decompositions, and factorizations of nonnegative
  matrices.
\newblock {\em Linear Algebra and its Applications}, 190:149--168, 1993.

\bibitem[FMP{\etalchar{+}}12]{fiorini2012lowerbound}
Samuel Fiorini, Serge Massar, Sebastian Pokutta, Hans~Raj Tiwary, and Ronald
  de~Wolf.
\newblock Linear vs. semidefinite extended formulations: Exponential separation
  and strong lower bounds.
\newblock In {\em Proceedings of the Forty-fourth Annual ACM Symposium on
  Theory of Computing}, STOC '12, pages 95--106. ACM, 2012.

\bibitem[GPT13]{gouveia2011lifts}
Jo{\~a}o Gouveia, Pablo~A. Parrilo, and Rekha~R. Thomas.
\newblock Lifts of convex sets and cone factorizations.
\newblock {\em Mathematics of Operations Research}, 38(2):248--264, 2013.

\bibitem[GRT13]{gouveia2013polytopes}
Jo{\~a}o Gouveia, Richard~Z. Robinson, and Rekha~R. Thomas.
\newblock Polytopes of minimum positive semidefinite rank.
\newblock {\em Discrete \& Computational Geometry}, 50(3):679--699, 2013.

\bibitem[LT12]{leetheis2012support}
T.~Lee and D.O. Theis.
\newblock Support-based lower bounds for the positive semidefinite rank of a
  nonnegative matrix, 2012.

\end{thebibliography}
\bibliographystyle{alpha}

\end{document}